\documentclass[final]{siamart251216}

\usepackage{showkeys}

\usepackage{tikz}
\usepackage{pgf}
\usetikzlibrary{arrows}
\usetikzlibrary{arrows.meta}
\usetikzlibrary{calc}
\usetikzlibrary{decorations.pathmorphing}
\usepackage{xfrac}    

\usepackage{faktor}
\usepackage{soul}
\setstcolor{red}
\setulcolor{red}

\usepackage{rotating} 

\usepackage{mathrsfs}
\DeclareMathAlphabet{\mathpzc}{OT1}{pzc}{m}{it}

\usepackage{color}
\usepackage{amsmath}
\usepackage{graphicx}
\usepackage{graphics}
\usepackage{float}
\usepackage{amsfonts}
\usepackage{amssymb,bbm}%
\usepackage{latexsym}
\usepackage{psfrag}
\usepackage{subfigure}
\usepackage{accents}

\newtheorem{remark}[theorem]{Remark}
\newtheorem{example}[theorem]{Example}

\newcommand{\cH}{{\cal H}}

\newcommand{\cD}{{\cal D}}

    \newcommand\quotient[2]{
        \mathchoice
            {
                \text{\raise1ex\hbox{$#1$}\Big/\lower1ex\hbox{$#2$}}%
            }
            {
                #1\,/\,#2
            }
            {
                #1\,/\,#2
            }
            {
                #1\,/\,#2
            }
    }

\newcommand{\beq}{\begin{equation}}
\newcommand{\eeq}{\end{equation}}
\newcommand{\beqs}{\begin{equation*}}
\newcommand{\eeqs}{\end{equation*}}
\newcommand{\bit}{\begin{itemize}}
\newcommand{\eit}{\end{itemize}}
\newcommand{\ben}{\begin{enumerate}}
\newcommand{\een}{\end{enumerate}}
\newcommand{\bal}{\begin{align}}
\newcommand{\eal}{\end{align}}
\newcommand{\bals}{\begin{align*}}
\newcommand{\eals}{\end{align*}}
\newcommand{\bse}{\begin{subequations}}
\newcommand{\ese}{\end{subequations}}
\newcommand{\bpr}{\begin{proposition}}
\newcommand{\epr}{\end{proposition}}
\newcommand{\bre}{\begin{remark}}
\newcommand{\ere}{\end{remark}}
\newcommand{\bpf}{\begin{proof}}
\newcommand{\epf}{\end{proof}}
\newcommand{\ble}{\begin{lemma}}
\newcommand{\ele}{\end{lemma}}
\newcommand{\bco}{\begin{corollary}}
\newcommand{\eco}{\end{corollary}}
\newcommand{\bex}{\begin{example}}
\newcommand{\eex}{\end{example}}
\newcommand{\bth}{\begin{theorem}}
\newcommand{\enth}{\end{theorem}}

\newcommand{\Rea}{\mathbb{R}}
\newcommand{\Com}{\mathbb{C}}

\def\XXint#1#2#3{{\setbox0=\hbox{$#1{#2#3}{\int}$}
     \vcenter{\hbox{$#2#3$}}\kern-.5\wd0}}

\usepackage{hyperref}
\counterwithin{figure}{section}
\definecolor{myblue}{rgb}{0,0,0.6}
\hypersetup{colorlinks=true,linkcolor=myblue,citecolor=myblue,filecolor=myblue,urlcolor=myblue}

\allowdisplaybreaks[4]

\newcommand{\tas}{\text{ as }}
\newcommand{\tand}{\text{ and }}

\newcommand{\vertiii}[1]{{\left\vert\kern-0.25ex\left\vert\kern-0.25ex\left\vert #1
    \right\vert\kern-0.25ex\right\vert\kern-0.25ex\right\vert}}

\definecolor{jwcol}{RGB}{27, 137, 18}  

\definecolor{dalcol}{rgb}{0.8,0,0}

\definecolor{escol}{rgb}{0,0,0.8}
\definecolor{estcol}{rgb}{0,0.5,0}
\definecolor{esnewcol}{rgb}{0,0.5,0}

\newcommand{\dist}{{\rm dist}}
\newcommand{\supp}{\operatorname{supp}}

\newcommand{\Op}{{\rm Op}}

\newcommand{\bb}{{\rm bb}}
\newcommand{\tr}{{\rm tr}}
\newcommand{\start}{{\rm st}}
\newcommand{\lin}{{\rm lin}}

\newcommand{\e}{\epsilon}

\newcommand{\mc}[1]{\mathcal{#1}}

\newtheorem{informaltheorem}[theorem]{Informal Theorem}

\makeatletter
\newcommand{\settheoremtag}[1]{
  \let\oldthetheorem\thetheorem
  \renewcommand{\thetheorem}{#1}
  \g@addto@macro\endtheorem{
    \addtocounter{theorem}{-1}
    \global\let\thetheorem\oldthetheorem}
  }
\makeatother

\usetikzlibrary{positioning}
\usepackage{lscape}

\definecolor{jeffColor}{RGB}{102, 0, 204}

\title{Cartesian PML truncation for the Helmholtz equation is exponentially accurate at high frequency}
\author{
Martin~Averseng\thanks{CNRS, Laboratoire Angevin de Recherche Mathématique, Universit\'e d'Angers, 49045 Angers, France, \tt martin.averseng@univ-angers.fr },
\and
Jeffrey~Galkowski\thanks{Department of Mathematics, University College London, London, WC1H 0AY, UK,   \tt J.Galkowski@ucl.ac.uk}
, 
\and
Euan~A.~Spence\thanks{Department of Mathematical Sciences, University of Bath, Bath, BA2 7AY, UK, \tt E.A.Spence@bath.ac.uk }
}
\date{\today}

\begin{document}
\maketitle

\begin{abstract}
For a wide class of Helmholtz scattering problems (namely, those fitting in the black-box framework of \cite{SjZw:91}) we prove that the error incurred by approximating the radiation condition by a Cartesian perfectly-matched layer (PML) decreases exponentially in 
the width of the PML, the strength of the PML scaling, and the frequency.
\end{abstract}



\section{Introduction}

Since the work of Berenger~\cite{Be:94}, perfectly matched layers (PMLs) have become a standard tool in the numerical simulation of frequency-domain wave problems such as the Helmholtz equation. This method approximates the solution of a scattering problem in an unbounded domain by making a complex change of variables in a layer away from the region of interest and truncating the problem with a zero Dirichlet condition. 

It has long been understood that for a radial PML, the error in the truncation decreases exponentially with the PML width and the strength of the PML scaling \cite[Theorem 2.1]{LaSo:98}, \cite[Theorem A]{LaSo:01}, \cite[Theorem 5.8]{HoScZs:03}, \cite[Theorem 3.4]{BrPa:07}. More recently, \cite[Theorem 1.5]{GLS2} proved that the error in the truncation decreases exponentially in the PML width, the strength of the scaling, \emph{and} the frequency. 

Cartesian PML is a very popular way of approximating the radiation condition in numerical-analysis applications using rectangular meshes.
However, Cartesian PML is more difficult to analyse than radial PML -- this is because the damping in the Cartesian PML operator only occurs in one coordinate direction at a time; i.e., a ray moving inside the PML and perpendicular to the direction of the scaling is not damped. 

The papers \cite{Ki:09, KiPa:10, KiPa:10a} proved that the error in Cartesian PML truncation is exponentially small in the PML width, with then \cite{BrPa:13} proving that the error is exponentially small in both the PML width and the strength of the scaling.

Regarding the frequency-dependence of the Cartesian PML truncation error:~using propagation of singularities (see, e.g., \cite[Appendix E]{DyZw:19}) it is relatively straightforward to show that the error in Cartesian PML truncation (and indeed, truncation with any complex absorption) is super-polynomially small in frequency (i.e., given $N>0$ there exists $C_N$ such that the error is $\leq C_N k^{-N}$ for all sufficiently large frequencies $k$). These arguments are written out -- and this error bound proved -- in the case of scattering by smooth variable coefficients (where these propagation arguments are the least technical) in \cite[Theorem 4.7]{GGGLS2}.
In the case of no scattering, 
and for Cartesian PML where the scaling function is linear at the PML boundary, 
the error was proved to be exponentially small in the PML width, the scaling strength, and the frequency in \cite[Lemma 3.4]{ChXi:13}.
Very recently, this result was extended to the case of scattering by a star-shaped Dirichlet obstacle \cite[Theorem 5.6]{wang2026wavenumber}.

The main result of the present paper is informally summarised as the following.

\begin{informaltheorem}
Consider a Helmholtz scattering problem fitting in the black-box framework of \cite{SjZw:91}; recall that this includes scattering by Lipschitz Dirichlet and Neumann obstacles and/or Lipschitz penetrable obstacles \cite[\S 2.2]{LSW1}. 

Suppose that the solution operator is polynomially bounded in frequency -- as occurs for ``most'' frequencies by \cite[Theorem 3.3]{LSW1}  in general, and always for nontrapping scatterers including star-shaped obstacles. 

Then, for Cartesian PML truncation where the scaling function is linear at the PML boundary, the error decreases exponentially in the width of the PML, the strength of the scaling, and the frequency.
\end{informaltheorem}

\section{Statement of the main result}

\subsection{Recap of the black-box framework}
We consider $P$ a (semiclassically scaled) Helmholtz scattering operator in the black-box framework of Sjostrand--Zworski \cite{SjZw:91}. We give a short recap here, and refer to \cite[Chapter 4]{DyZw:19} for the details.
In this framework, $P= -k^{-2}\Delta$ outside a ball and is equal to some self-adjoint operator inside the ball. Black-box operators include 
scattering by Dirichlet, Neumann, and penetrable obstacles (even for Lipschitz boundaries -- see \cite[\S2.2]{LSW1}), and scattering by inhomogeneous media. 

For $i=1,\ldots, d$, let 
$R_{\bb,i}$ be such that the black box is contained in 
$Q_{\bb}:=\prod_{i=1}^d (-R_{\bb,i},R_{\bb,i})$. Let 
$$
\cH:=\cH_0\oplus L^2(\mathbb{R}^d\setminus Q_{\bb}),
$$
where $\cH_0$ is a Hilbert space such that $P:\cH\to \cH$ is an unbounded operator with domain $\cD$. (Recall that if $u=(0,u_1)$ with $\supp u_1\cap \overline{Q_{\bb}}=\emptyset$, then $u\in\mc{D}\Leftrightarrow u_1\in H^2(\mathbb{R}^d\setminus Q_{\bb})$ and, furthermore, $P(0,u_1)=(0,-k^{-2}\Delta u_1)$.) We use the following notation for multiplication operators on $\cH$: for $u=(u_1,u_2)\in \cH$ and $\chi\in L^\infty(\mathbb{R}^d)$ with $\supp (1-\chi)\cap Q_{\bb}=\emptyset$ we set  $\chi u:=(u_1,\chi u_2)$. 

In addition to the usual black-box assumptions, we require the following unique-continuation-type property: 
\beq\label{e:ucp}
\text{ if } u=(u_1,u_2)\in\cD,\, (P-1)u=0,\text{ and $u_2$ is compactly supported, then $u=0$}
\eeq
(note that this is the case for the scattering examples listed above; see the discussion and references in \cite[\S4.2]{GSAN}). This unique-continuation property combined with Rellich's theorem (see \cite[Theorem 4.17]{DyZw:19}) implies uniqueness of outgoing solutions to the equation $(P-1)u= g$ for compactly supported $g\in \cH$.

We can now define the limiting absorption resolvent $R_P:= \lim_{\e\to 0^+}(P- 1-i\e)^{-1}$  and recall that $\chi R_P \chi : \cH\to \cD$ \cite[Theorem 4.4]{DyZw:19}. 

\subsection{Definition of the scaled operators}\label{s:def}

Given $R_{\bb, j}$, $j=1,\ldots,d$, let $R_{\bb, j}< R_{\start,j} < R_{\lin, j}< R_{\tr, j}$ (where ``$\bb$'' stands for ``black-box'', ``$\start$" stands for ``start", ``$\lin$" stands for ``linear", and ``$\tr$" stands for ``truncation"), and let $Q_{\star} := \prod_{j = 1}^d (-R_{\star,j},R_{\star,j})$ with $\star$ standing for either of the previous subscripts. 
We consider a Cartesian totally real deformation (see \cite[Section 4.5]{DyZw:19}), $\Gamma$, of $\mathbb{R}^d$ given by
\beq\label{e:gamma}
\gamma_\theta:x\mapsto x+i(\tan \theta) \nabla F(x),\quad \text{ where } \quad
F(x):=\sum_{j=1}^d F_j(x_j),
\eeq
with $F_{j}\in C^\infty(\mathbb{R})$, $F_j\equiv 0$ on $[-R_{\bb,j},R_{\bb,j}]$ and $F_j'(x)= x$ for $|x|\geq R_{\lin,j}$, $\{F_j''>0\} =\{|x|>R_{\start,j}\}$
and $F_j$ convex. 
(The notation ``\lin" is used since when $|x_j|\geq R_{\lin,j}$ the scaling/stretching in the $j$th coordinate is linear by \eqref{e:gamma}.)
We identify $\Gamma$ with $\mathbb{R}^d$ using the map $\gamma_\theta$.

Let 
$$
\Delta_\theta:= \sum_{j=1}^d \bigg(
\frac{1}{1+i(\tan \theta) F''_j(x_j)}\partial_{  x_j}
\bigg)^2,
$$
and let
$$
P_{\theta,\rm free}-1 := -k^{-2}\Delta_\theta-1;
$$ 
i.e., $P_{\theta,\rm free}-1 $ is the free, scaled Helmholtz operator. 
Let
$
(P_{\theta}-1)
$
be the operator $P$ deformed to $\Gamma$, and let 
$
(P^D_\theta-1)
$
be the Dirichlet realization of $P_\theta-1$ on $Q_{\tr}$
More precisely, 
let 
\begin{equation*}
\begin{gathered}
\mc{H}(Q_{\tr}):=\mc{H}_{0}\oplus L^2(Q_{\tr} \setminus Q_{\bb}),\\
\mc{D}(Q_{\tr}):=\Big\{u\in \mc{H}(Q_{\tr}) \,:\, 
{\text{for all}\,\chi\in C_c^\infty(Q_{\start}),\,\chi\equiv 1\text{ on }Q_{\bb},
\chi u\in \mc{D},
}\\
\hspace{5cm} 
 { (1-\chi)u \in H_0^1(Q_{\tr}),\, -\Delta_\theta ((1-\chi)u)\in L^2(Q_{\tr})}\Big\},
\end{gathered}
\end{equation*}
and then define
$$
P^{D}_\theta u:= P(\chi u)+ (-k^{-2}\Delta_\theta)((1-\chi )u)
$$
so that 
$P^{D}_\theta:\mc{H}(Q_{\tr})\to \mc{H}(Q_{\tr})$ with domain $\mc{D}(Q_{\tr})$
and norm
\beq\label{e:Dnorm}
\|u\|_{\mc{D}(Q_{\tr})}^2=\|u\|_{\mc{H}(Q_{\tr})}^2+\|P_\theta u\|_{\mc{H}(Q_{\tr})}^2,\qquad u\in \mc{D}(Q_{\tr}).
\eeq

\subsection{The main result}

\begin{theorem}[Exponential accuracy of Cartesian PML truncation]\label{t:main}
Let $\chi\in C_c^\infty( Q_{\start})$ with 
 $\supp(1-\chi)\cap Q_{\bb}=\emptyset$,
and suppose that there are $\mathcal{J}\subset (0,\infty)$ and $M>0$ such that for all $k_0>0$, there is $C>0$ such that for $k\in\mathcal{J}$, $k \geq k_0$
$$
\|\chi R_P\chi\|_{\cH\to \cH}\leq Ck^M.
$$
Then, for all $\delta,\e>0$ there are $C,c,k_1>0$ such that for 
$R_{\start,i}+\e<R_{\lin}$, $R_{\lin}+\e<R_{\tr,i}$, for all 
$\theta\in[\delta,\pi/2-\delta]$, and all $k\in \mathcal{J}$ with $k\geq k_1$,
$(P^D_\theta-1)^{-1}: \cH(Q_{\tr})\to \cD(Q_{\tr})$ exists with
\beq\label{e:main}
\big\|\chi \big(R_P-(P^D_\theta-1)^{-1}
\big)\chi\big\|_{\cH(Q_{\tr})\to \cD(Q_{\tr})}\leq C\exp\big(- ck\tan \theta\min_i\{R_{\tr,i}-R_{\start,i}\}
\big).
\eeq
\end{theorem}

We emphasise that, in the case of 
scattering by Dirichlet, Neumann, and penetrable Lipschitz obstacles and scattering by variable coefficients, for all $k_0>0$, there is $C>0$ such that for $k>k_0$, 
$$
\|u\|_{H_k^1}^2:=\|u\|^2_{L^2}+\|k^{-1}\nabla u\|^2_{L^2}\leq C\|u\|^2_{\cD(Q_\tr)};$$
i.e., \eqref{e:main} controls the $k$-weighted $H^1$ error. 

\paragraph{Overview of the proof}


 Let $g \in \cH$, $\chi$ a cutoff supported in the unscaled region $Q_{\start}$, and consider $u_\theta,u_\theta^D$ the solutions of
$$(P_\theta - 1)u_\theta = (P_\theta^D - 1) u_\theta^D = \chi g.$$
By the well-known property that $u_\theta$ equals the scattering solution $R_P\chi g$ in the unscaled region (see Theorem \ref{t:agree} below), it suffices to show that $u_\theta$ and $u_\theta^D$ agree up to exponentially small errors in the unscaled region. 

To bound the error $E:= \psi u_\theta-u_\theta^D$ (where $\supp \chi \subset \{\psi \equiv 1\} \subset \supp \psi \subset Q_{\tr}$), we observe that it satisfies the PDE
$$(P_\theta^D-1) E = [P_\theta,\psi] u_\theta = [P_\theta,\psi] \phi u_\theta$$
for some cutoff $\phi \equiv 0$ near the black-box and $\phi \equiv 1$ outside of $Q_{\start}$. 
We now use the fact that any Helmholtz solution away from the scatterer can be written in terms of the free resolvent. Indeed, 
$$ (P_{\theta,{\rm free}} - 1)\phi u_\theta =  (\phi \chi g + [P_{\theta,\rm {free}},\phi]u_\theta)=:\tilde{g} \qquad \Longrightarrow\qquad \phi u_\theta= (P_{\theta,{\rm free}} - 1)^{-1} \widetilde{g}\,,$$
and thus 
$$E = (P_\theta^{D}-1)^{-1} [P_\theta,\psi] (P_{\theta,{\rm free}}-1)^{-1} \widetilde{g}.$$
Importantly, $\widetilde{g}$ is supported in the unscaled region, and $[P_\theta,\psi]$ is supported near the truncation boundary.

Theorem~\ref{t:main}
then 
follows from two additional properties.
\begin{enumerate}
\item Exponential decay  of $\chi_1(P_{\theta,{\rm free}}-1)^{-1}\chi_2$ with respect to the distance between the supports of $\chi_1$ and $\chi_2$.
\item An estimate on $(P_\theta^D-1)^{-1}$ of the form $\|(P_\theta^D - 1)^{-1}\|\leq Ck^M$.

\end{enumerate}


Property 1 follows from the decaying behaviour of the free Green's function in the complex-scaled coordinates (see Lemma \ref{l:decay}).

Property 2 follows by first bounding $(P_\theta^D-1)^{-1}$ in terms of $(P_\theta-1)^{-1}$ (see Lemma \ref{l:inherit0}), and then bounding the latter in terms of $\chi R_P\chi$ (see Lemma \ref{l:inherit}), which is polynomially bounded by assumption. This first bound requires an estimate on $(P_{\theta,\rm free}-1)^{-1}$ (see Theorem \ref{t:free}).

The bound on $(P_\theta^D-1)^{-1}$ in terms of $(P_\theta-1)^{-1}$ follows from an a priori estimate on solutions to 
$$
(P_\theta^D-1)u=g
$$
in a neighborhood of the truncation boundary.  
For radial PMLs, this type of estimate follows easily from coercivity of the quadratic on functions supported near the truncation boundary (see \cite[Lemma 4.4]{GLS2}). However, because Cartesian PMLs are not coercive near the truncation boundary (coercivity only occurs in the ``corners'' of the scaling region), this estimate is now more delicate  and relies on a propagation of semiclassical singularities argument. 
Namely, 
at any point $(x,\xi)$ in phase space with $x$ near the truncation boundary, \emph{either} the semiclassical symbol of the operator is non-zero -- and hence the operator is elliptic --\emph{or} the backward Hamiltonian trajectory escapes to a region of ellipticity in finite time; see Figure \ref{fig:Martin_is_better_than_chatGPT}.

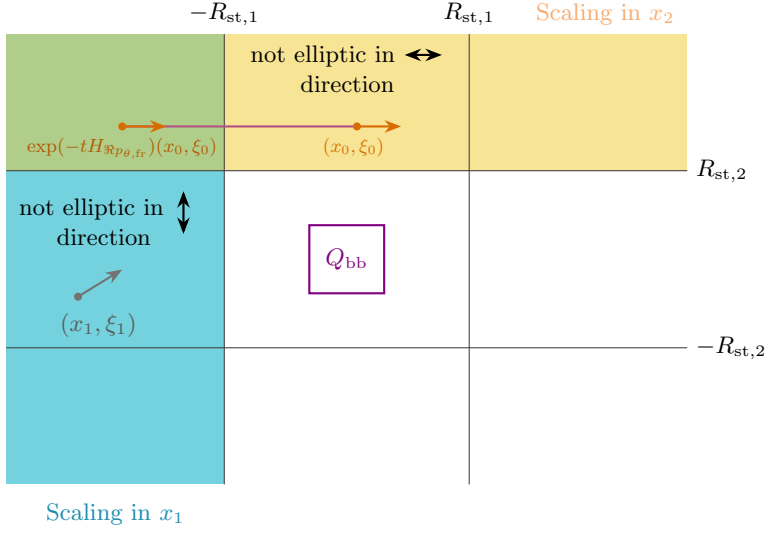
\begin{figure}
    \centering
    \begin{tikzpicture}[x=0.9cm,y=0.9cm,font=\small,
  >={Stealth[length=2.2mm]},line cap=round,
  interface/.style={black!75,thin},
  momentum/.style={->,thick,orange!85!black}]
  \definecolor{scalecyan}{RGB}{115,210,222}
  \definecolor{scaleyellow}{RGB}{247,228,148}
  \definecolor{scalegreen}{RGB}{177,205,138}
  \fill[scalecyan] (-5,-3.3) rectangle (-1.8,3.3);
  \fill[scaleyellow] (-5,1.3) rectangle (5,3.3);
  \fill[scalegreen] (-5,1.3) rectangle (-1.8,3.3);

  \draw[interface] (-1.8,-3.3) -- (-1.8,3.3);
  \draw[interface] (1.8,-3.3) -- (1.8,3.3);
  \draw[interface] (-5,1.3) -- (5,1.3);
  \draw[interface] (-5,-1.3) -- (5,-1.3);
  \node[above] at (-1.8,3.3) {$-R_{\mathrm{st},1}$};
  \node[above] at (1.8,3.3) {$R_{\mathrm{st},1}$};
  \node[right] at (5,1.3) {$R_{\mathrm{st},2}$};
  \node[right] at (5,-1.3) {$-R_{\mathrm{st},2}$};
  \node[above,scaleyellow!75!red] at (3.8,3.3) {Scaling in $x_2$};
  \node[below,text=cyan!65!black] at (-3.4,-3.45) {Scaling in $x_1$};

  \draw[violet,thick] (-0.55,-0.5) rectangle (0.55,0.5);
  \node[text=violet] at (0,0) {$Q_{\mathrm{bb}}$};

  \node[align=center] at (-0.35,2.8) {not elliptic in\\$\qquad$direction};
  \draw[<->,thick] (0.85,3) -- (1.4,3);
  \node[align=center] at (-3.75,0.55) {not elliptic in\\\quad direction};
  \draw[<->,thick] (-2.4,0.37) -- (-2.4,1);

  \draw[magenta!70!black,thick] (-3.3,1.95) -- (0.15,1.95);
  \fill[orange!85!black] (0.15,1.95) circle (1.5pt);
  \draw[momentum] (0.15,1.95) -- (0.8,1.95);
  \node[below,text=orange!85!black,scale=.8] at (0.1,1.9) {$(x_0,\xi_0)$};
  \fill[orange!85!black] (-3.3,1.95) circle (1.5pt);
  \draw[momentum] (-3.3,1.95) -- (-2.65,1.95);
  \node[below,text=orange!70!black,scale =.8] at (-3.3,1.9)
    {$\exp(-tH_{\Re p_{\theta,\rm{fr}}})(x_0,\xi_0)$};

  \fill[black!55] (-3.95,-0.55) circle (1.5pt);
  \draw[->,thick,black!55] (-3.95,-0.55) -- (-3.3,-0.15);
  \node[below,text=black!65] at (-3.65,-0.65) {$(x_1,\xi_1)$};
\end{tikzpicture}

    \caption{Sketch of the main argument used to prove an a priori estimate for solutions of the Cartesian PML problem near the truncation boundary.  $(x_1,\xi_1)$ is a point of phase-space where the problem is elliptic and hence estimates are available without propagation. In contrast, $(x_0,\xi_0)$ is a point where the problem is not elliptic, but from any such point, the backward Hamiltonian flow associated to the real part of the symbol of $P_{\theta}-1$ will eventually reach a corner (i.e. a point of ellipticity)
 }
    \label{fig:Martin_is_better_than_chatGPT}
\end{figure}


\begin{remark}
    We highlight that the only places where the Cartesian structure of the PML is used in the above arguments are 
    (i) to bound $(P_{\theta,\rm free}-1)^{-1}$,
    (ii) to prove the exponential estimate on the free Green's kernel, and (iii) to prove the a priori estimate near the truncation boundary. In particular, in any scaling where one can establish these two estimates one can prove exponential error bounds for the PML problem.
\end{remark}






\paragraph{Why we assume that $F'_i$ is linear near the PML boundary}
The principal symbol of the scaled operator $P_\theta$ has a nonpositive imaginary part; see \eqref{e:symbolSign} below.
For such an operator, the relevant propagation of singularities results in the semiclassical calculus for domains with a boundary have not yet been written down in the literature. Indeed, such propagation of singularities results in the semiclassical calculus have
been proved (i) on manifolds without boundary for operators with symbols whose imaginary parts are single-signed \cite[Theorem E.47]{DyZw:19} and (ii)
on manifolds with boundary for $P$ \cite{Va:04}. 

The assumption in \S\ref{s:def} that $F_i'(x)= x$ for $|x|\geq R_{\lin,i}$  allows us to use a reflection argument (similar to that used for the Cartesian PML analyses in 2-d in \cite[Proof of Theorem 5.5]{BrPa:13}, \cite[Lemma 3.4]{ChXi:13}) to extend the solution past $\partial Q_{\tr}$ and use the propagation results of (i) above, i.e., bypassing the issue of propagation up to the boundary. Once the relevant propagation results are written down in the literature, this assumption can be removed.





\section{Results about the free complex-scaled resolvent}
With  
\beq\label{e:Fourier}
\mathcal{F}_k(u)(\xi) := \int_{\mathbb{R}^d} e^{-i k\langle\xi, x\rangle} u(x) d x,
\eeq
 define for $s\in\mathbb{R}$
$$
\|u\|_{H_k^{s}(\mathbb{R}^d)}^2:=\frac{k^d}{(2\pi )^d}\int_{\Rea^d}(1+|\xi|^{2})^{s}|\mathcal{F}_k(u)(\xi)|^2 d\xi.
$$

\begin{theorem}[Bound on the free scaled resolvent]
\label{t:free}
Given $\delta>0$, $s\in\mathbb{R}$, there exists $C>0$ and $k_0>0$ such that for $k\geq k_0$ and $\theta\in[\delta,\frac{\pi}{2}-\delta]$,
$$
\big\|(P_{\theta,\rm free}-1)^{-1}\big\|_{H^s_k(\Rea^d) \to H^{s+2}_k(\Rea^d)}\leq C k.
$$
\end{theorem}

\bre
\cite[Theorem 3.6]{wang2026wavenumber}
proves an $L^2\to L^2$ bound analogous to that in 
Theorem \ref{t:free} for the case of a piecewise-linear scaling function.
The analogue of Theorem \ref{t:free} for radial PML is \cite[Theorem 3.2]{GLS2}. 
\ere

To prove Theorem \ref{t:free}, and also other results below, we use the following facts about the adjoint of $-\Delta_\theta$. Let 
$$
\widetilde{\Delta}_\theta:= \det (I+ i (\tan \theta)F''(x)) \Delta_\theta = \prod_{j=1}^d (1 + i (\tan \theta)F_j''(x_j))
\Delta_\theta.
$$
Then, by a simple calculation, 
\beqs
(\widetilde{\Delta}_\theta)^*= J \widetilde{\Delta}_\theta J, 
\eeqs
where $J$ is the action of complex conjugate; i.e., $J v = \overline{v}$. Therefore
\beq
\label{e:adjoint}
(\Delta_\theta)^* = 
J \det (I+ i(\tan \theta) F''(x)) \Delta_\theta J
(\det (I - i(\tan \theta)F''(x)))^{-1}.
\eeq

To prove Theorem \ref{t:free} we use 
semiclassical propagation results (i.e., PDE bounds describing how behaviour of Helmholtz solutions are dictated by the behaviour of the rays) from \cite[Appendix E]{DyZw:19} and \cite{Zw:12}. Note that here we use the small parameter $k^{-1}$, but \cite{DyZw:19, Zw:12} denote the small parameter by $h$. 

The semiclassical principal symbol of $P_{\theta,\rm{free}}$ is
$$
p_{\theta,\rm{fr}}(x,\xi):=\sigma(P_{\theta,\rm{fr}})=\sum_{i=1}^d\bigg(\frac{\xi_i}{1+i\tan \theta F''_i(x_i)}\bigg)^2.
$$
Observe that 
\begin{align}\nonumber
&\Re p_{\theta,\rm {fr}}(x,\xi)= \sum_{i=1}^d\frac{(1-(\tan \theta F_i''(x_i))^2)\xi_i^2}{\big(1+(\tan \theta F_i''(x_i))^2\big)^2},
\\
&\Im p_{\theta,\rm{fr}}(x,\xi)=\sum_{i=1}^d\frac{-2\tan \theta F_i''(x_i)\xi_i^2}{\big(1+(\tan \theta F_i''(x_i))^2\big)^2}\leq 0.
\label{e:symbolSign}
\end{align}

Before proceeding to estimate the free resolvent, we need a dynamical fact 
about the rays associated to $P_{\theta, \rm fr}$ (more precisely, the Hamiltonian flow of $\Re p_{\theta,\rm{fr}}$).

\begin{lemma}[Uniform escape of the rays to ellipticity]
\label{l:flow}
    Let $\delta>0$. Then there are $T>0$, $\mu>0$ such that for all $
    \theta\in[\delta,\frac{\pi}{2}-\delta]$, $(x_0,\xi_0)\in T^*\mathbb{R}^d$ with $|p_{\theta,\rm{fr}}(x_0,\xi_0)-1|\leq \mu$,  there is $0\leq t\leq T$ such that 
    $$
    |p_{\theta,\rm {fr}}(\exp(-tH_{\Re p_{\theta,\rm fr}}(x_0,\xi_0))-1|\geq \mu.
    $$
\end{lemma}
\begin{proof}
Define 
$$p^i_{\theta}:=\frac{1-(\tan \theta F_i''(x_i))^2}{(1+(\tan \theta F_i''(x_i))^2)^2}\xi_i^2.$$
Then, since $\{p^i_{\theta},\Re p_{\theta,\rm {fr}}\}=0$, $p^i_{\theta}$ is preserved by the Hamiltonian flow for $\Re p_{\theta,\rm {fr}}$.  
   
    Define $(x(t),\xi(t)):=\exp(tH_{\Re p_{\theta,\rm {fr}}})(x_0,\xi_0)$. Let 
    $$
    T_0:=-\sup\big\{ t\leq 0 \,:\, |p_{\theta,\rm {fr}}(x(t),\xi(t))-1|\geq \mu\big\}.
    $$
    Then,
\begin{equation} 
\label{e:flowEquations}
\begin{aligned}
    \partial_t( x_i\xi_i)&= 2 p^i_{\theta} +\frac{2\tan ^2\theta F_i''(x_i)(3- (\tan \theta F_i''(x_i))^2)F_i'''(x_i)}{(1+(\tan\theta F_i''(x_i))^2)^3}\xi_i^2x_i.
    \end{aligned}
    \end{equation}
Observe that on $|p_{\theta,\rm fr}-1|\leq \mu$, 
    $$
    0\leq F_i''(x_i)\xi_i^2\leq C_{\delta}|\Im (p_{\theta,\rm {fr}}-1)|=O_{\delta}(\mu)\qquad \Longrightarrow\qquad \Re p_{\theta,\rm fr}(x_0,\xi_0)= 1+O_{\delta}(\mu).
    $$
    Therefore, there is $i$ such that 
    $$
    p_\theta^i(x_0,\xi_0)=p_\theta^i(x(t),\xi(t))\geq \frac{1}{d}+O_{\delta}(\mu)\qquad \Longrightarrow \qquad |\xi_i(t)|\geq c+O_{\delta}(\sqrt{\mu})
    $$
    Hence, if $t\leq 0$ and $|x_i(t)|\geq R_{\lin, i}$, we have
    $$
    \Im p_{\theta,\rm {fr}}\leq -c<0
    $$
    which implies $t\leq-T_0$.  
    In particular, for $-T_0\leq t\leq 0$, $|x_i|\leq R_{\lin, i}$ and hence using this in~\eqref{e:flowEquations} along with Gronwall's inequality,  for $-T_0\leq t\leq 0$, 
    $$
    x_i(t)\xi_i(t)\leq \Big(-\frac{2}{d}+O_{\delta}(\mu^{1/2})\Big)|t| +x_i(0)\xi_i(0).
    $$
    But then, using again that $c\leq|\xi_i(t)|\leq C$, and $|x_i(t)|\leq R_{\lin, i}$, for $-T_0\leq t\leq 0$,
    $$
    R_{\lin, i}\geq |x_i(-T_0)|\geq C\bigg(\bigg(\frac{2}{d}-O_{\delta}(\mu^{1/2})\bigg)|T_0|-R_{\lin, i}\bigg),
    $$
    which implies that $|T_0|\leq C R_{\lin, i}$, completing the proof.
\end{proof}

\

Recall (from, e.g., \cite[\S 9.3]{Zw:12}) that 
$$
S^m(\mathbb{R}^{2d}):=\Big\{ a\in C^\infty(\mathbb{R}^{2d})\,:\,  
\forall\alpha,\beta \in \mathbb{N}^d,
\sup_{\substack{(x,\xi)\in\mathbb{R}^{2d}\\ k>1
}
}\langle \xi\rangle^{-m+|\beta|}|\partial_x^\alpha \partial_{\xi}^\beta a(x,\xi)|<\infty\Big\},
$$
where $\langle\xi\rangle:= (1+|\xi|^2)^{1/2}$.
Furthermore, for $u\in \mc{S}'$ and $a\in S^m$, we define
$$
\Op_k(a)u(x):=\frac{k^d}{(2\pi )^d}\int_{\Rea^d}\int_{\Rea^d} e^{ik\langle x-y,\xi\rangle}a(x,\xi)u(y)dyd\xi.
$$

\begin{lemma}[Elliptic estimate]
\label{l:elliptic}
    Let $\mu,\delta, N, k_0>0$, $s\in\mathbb{R}$ and $a\in S^0$ with $\supp a\subset \{|p_{\theta,\rm fr}-1|\geq \mu\}$. Then there is $C>0$ such that for all $\theta\in [\delta,\frac{\pi}{2}-\delta]$,   $u\in \mathcal{S}'(\mathbb{R}^d)$ with $(P_{\theta,\rm {free}}-1)u \in H^{s}_k(\Rea^d)$, and $k\geq k_0$,
    $$
    \|\Op_k(a)u\|_{H_k^{s+2}}\leq C\|(P_{\theta,\rm {free}}-1)u\|_{H_k^{s}} +C_Nk^{-N}\|u\|_{H_k^{-N}}.
    $$
\end{lemma}
\begin{proof}
    Observe that for all $\mu>0$, there is $c>0$ such that 
    $$
    \big\{ |p_{\theta,\rm {fr}-1}-1|\geq \mu\big\}\subset \big\{ |p_{\theta,\rm{fr}}-1|\geq c\langle \xi\rangle^2\big\}.
    $$
    Hence, the lemma follows from the elliptic estimate for semiclassical pseudodifferential operators (see e.g. the proof of~\cite[Theorem E.33]{DyZw:19} together with~\cite[Theorem 9.5]{Zw:12}).
\end{proof}

\begin{lemma}[Propagation estimate]
\label{l:propagateFree}
    Let $k_0>0$, there is $\mu>0$ such that for all $\delta, N>0$ and $a\in S^0$ with $\supp a\subset \{|p_{\theta,\rm fr}-1|\leq \mu\}$, there is $C>0$ such that for all $\theta\in [\delta,\frac{\pi}{2}-\delta]$,   $u\in \mathcal{S}'(\mathbb{R}^d)$ with $(P_{\theta,\rm {free}}-1)u \in H^{-N}_k(\Rea^d)$, and $k\geq k_0$,
    $$
    \|\Op_k(a)u\|_{H_k^N(\Rea^d)}\leq Ck\|(P_{\theta,\rm {free}}-1)u\|_{H_k^{-N}(\Rea^d)} +C_Nk^{-N}\|u\|_{H_k^{-N}(\Rea^d)}.
    $$
\end{lemma}
\begin{proof}
It will be convenient in this proof to work with pseudodifferential operators whose kernels are supported in a unit  neighborhood of the diagonal. To do this, we fix $\psi\in C_c^\infty([-1,1])$ with $\supp (1-\psi)\cap [-1/2,1/2]=\emptyset$, and write for $b\in S^\infty$, and $u\in \mc{S}'$, 
    $$
    \widetilde{\Op}_k(b)u(x):=\frac{k^d}{(2\pi )^d}\int_{\Rea^d}\int_{\Rea^d}e^{ik\langle x-y,\xi\rangle} b(x,\xi)\psi(|x-y|)u(y)dyd\xi.
    $$
    By integration by parts in $\xi$ (see \cite[Theorem 9.6]{Zw:12}),
    $$
    \|(\Op_k(a)-\widetilde{\Op}_k(a))u\|_{H_k^N}\leq C_Nk^{-N}\|u\|_{H_k^{-N}}.
    $$
        Therefore, it is sufficient to estimate $\widetilde{\Op}_k(a)u$. For this, let $\chi \in C_c^\infty(B(0,4\sqrt{d});[0,1])$ such that $\sum_{n\in\mathbb{Z}^d}\chi(x-n)\equiv 1$ and define
$$
a_n:=a(x,\xi)\chi_n(x),\qquad \chi_n(x):=\chi(x-n).
$$

 Then, by~\cite[Theorem E.47]{DyZw:19} together with locality of $P_{\theta,\rm{free}}$ and Lemma~\ref{l:flow} there are $C>0$, $b_n,b_n'\in C_c^\infty(\mathbb{R}^{2d})$, $\phi_n\in C_c^\infty(\mathbb{R}^d)$ with $\supp b_n\subset \{ |p_{\theta,\rm{fr}}-1|\geq c\langle \xi\rangle^2,\, |x-n|\leq C\}$, $\supp b_n'\subset \{ |x-n|\leq C\}$, $\supp \phi_n\subset\{ |x-n|\leq C\}$ such that 
   \begin{align*}
&\|\widetilde{\Op}_k(a_n)u\|_{L^2}\\
&\quad\leq Ck\|\widetilde{\Op}_k(b_n')(P_{\theta,\rm {free}}-1)u\|_{H_k^{-1}}+C\|\widetilde{\Op}_k(b_n)u\|_{L^2}+C_Nk^{-N}\|\phi_n u\|_{H_k^{-N}}\\
    &\quad\leq Ck\|\phi_n (P_{\theta,\rm {free}}-1)u\|_{H_k^{-N}}+C\|\widetilde{\Op}_k(b_n)u\|_{L^2}+C_Nk^{-N}\|\phi_n u\|_{H_k^{-N}}.
   \end{align*}
   Next, using the elliptic estimate (see, e.g., the proof of~\cite[Theorem E.33]{DyZw:19} together with~\cite[Theorem 9.5]{Zw:12}) and locality of $P_{\theta,\rm {free}}$ to control $\widetilde{\Op}_k(b_n)u$, we obtain 
   \begin{equation} 
   \label{e:localEstimate}
   \begin{aligned}
    \|\widetilde{\Op}_k(a_n)u\|_{L^2}
    &\leq Ck\|\phi_n (P_{\theta,\rm {free}}-1)u\|_{H_k^{-N}}+C_Nk^{-N}\|\phi_n u\|_{H_k^{-N}}.
   \end{aligned}
   \end{equation}
   Now, we estimate
   \begin{equation}\label{e:hightoLow}
       \|\widetilde{\Op}_k(a)u\|_{H_k^N}^2\leq C\|\widetilde{\Op}_k(a)u\|_{L^2}^2+C_Nk^{-N}\|u\|_{H_k^{-N}}^2.
\end{equation}
Then, using the support properties of $\phi_n$,  and~\eqref{e:localEstimate}, we have
\begin{align}\nonumber
       \|\sum_{n\in \mathbb{Z}^d}\widetilde{\Op}_k(a_n)u\|_{L^2}^2
       &=\sum_{n,n'}\langle \widetilde{\Op}_k(a_n)u,\widetilde{\Op}_k(a_{n'})u\rangle_{L^2}\\ \nonumber
       &=\sum_{|n-n'|\leq C}\langle \widetilde{\Op}_k(a_n)u,\widetilde{\Op}_k(a_{n'})u\rangle_{L^2}\\ \nonumber
       &\leq C\sum_n\|\widetilde{\Op}_k(a_n)u\|_{L^2}^2\\ \nonumber
       &\leq C\sum_n k^{2}\|\phi_n (P_{\theta,\rm {free}}-1)u\|^2_{H_k^{-N}}+C_Nk^{-N}\|\phi_n u\|_{H_k^{-N}}^2\\
       &\leq Ck^{2}\|(P_{\theta,\rm {free}}-1)u\|^2_{H_k^{-N}}+C_Nk^{-N}\|u\|_{H_k^{-N}}^2.\label{e:orthogonal}
   \end{align}
Combining~\eqref{e:hightoLow} with~\eqref{e:orthogonal} completes the proof.
\end{proof}

\

\begin{proof}[Proof of Theorem \ref{t:free}]
    Let $\mu$ be as in Lemma~\ref{l:propagateFree} and $a\in S^0$ with $\supp a\subset \{ |p_{\theta,\rm{fr}}-1|\leq \mu\}$ and $\supp (1-a)\subset \{ |p_{\theta,\rm{fr}}-1|\geq \frac{\mu}{2}\}.$ Then, Lemmas~\ref{l:propagateFree} and~\ref{l:elliptic} imply, respectively, that 
    $$
    \|\Op_k(a)u\|_{H_k^N}\leq Ck\|(P_{\theta,\rm {free}}-1)u\|_{H_k^{-N}}+C_Nk^{-N}\|u\|_{H_k^{-N}}
    $$
    and 
    $$
    \|\Op_k(1-a)u\|_{H_k^{s+2}}\leq C\|(P_{\theta,\rm{free}}-1)u\|_{H_k^{s}}+C_Nk^{-N}\|u\|_{H_k^{-N}}.
    $$
    Hence, for $k$ large enough,
    $$
    \|u\|_{H_k^{s+2}}\leq Ck\|(P_{\theta,\rm{free}}-1)u\|_{H_k^{s}}.
    $$
    An identical argument but applied to $P_{\theta,\rm{free}}^*$ 
(using the relation \eqref{e:adjoint})
    implies
      $$
    \|u\|_{H_k^{-s}}\leq Ck\|(P^*_{\theta,\rm{free}}-1)u\|_{H_k^{-s-2}},
    $$
    and hence that $(P_{\theta,\rm{free}}-1)$ is invertible with the required estimate.
\end{proof}

\begin{lemma}[Decay estimates for the free complex-scaled resolvent]\label{l:decay}
Given $\delta,\e,C_0,k_0>0$
there exists $C,\widetilde{c}>0$ such that the following is true.
Let $\chi_1,\chi_2\in C^\infty(\Rea^d;[0,1])$ with $\supp \chi_1 \subset Q_{\tr}\setminus Q_{\lin}$
 $\supp \chi_2 \subset Q_{\start}$, and
 $$
 \|\partial^\alpha \chi_i\|_{L^\infty}\leq C_0,\qquad |\alpha|\leq 2, \quad i=1,2.
 $$
Then for $R_{\start}+\e < R_{\lin}$, $k>k_0$, $\theta\in [\delta,\pi/2-\delta]$,
$$
\big\| \chi_1 (P_{ \theta,\rm free}-1)^{-1}\chi_2 \big\|_{L^2(\Rea^d) \to H_k^2(\Rea^d)} \leq C 
\exp\big(- \widetilde{c} k\tan\theta\,
\dist(\supp \chi_1 ,\supp \chi_2)\big).
$$
\end{lemma}

To prove Lemma \ref{l:decay}, we use the following lemma.

\begin{lemma}\label{l:sneaky}
Suppose that $u, v\in \Rea^d$ with $\langle u,v\rangle \geq 0$ and $u\neq 0$. 
Then (with the principal branch of the square root),
$$
\Im\Big( \Big(\sum_{j=1}^d (u_j+ i v_j)^2 \Big)^{1/2}\Big)
 \geq \frac{\langle u,v\rangle}{|u|}.
 $$
\end{lemma}

\bpf
To keep the notation concise, we denote $\sum_{j=1}^d (u_j+ i v_j)^2$ by $\langle u+ iv , u+i v\rangle$ (i.e., the inner product on $\Rea^d$ is extended to $\Com^d$ bilinearly). 
Let $p,q$ be such that 
\beq\label{e:rain1}
p+i q:=\langle u+ iv ,u+iv\rangle^{1/2}
\quad 
\text{ so that }
\quad (p+iq)^2=
|u|^2-|v|^2+2i\langle u,v\rangle.
\eeq
Now $q\geq 0$ since $\langle u,v\rangle \geq 0$, and $p\geq 0$ from the principal branch of the square root. 

We need to prove that $q \geq \langle u, v\rangle/|u|.$
Taking the imaginary part of the second equation in  \eqref{e:rain1} we see that  $pq = \langle u,v\rangle$, so it is enough to show that $p\leq |u|$. 
Taking the real part of the second equation in \eqref{e:rain1} and combining it with the equality $p^2+q^2 = |\langle u+ iv , u+i v\rangle|$, we obtain that 
\[
p^2
=
\frac{
\left|\langle u+iv,u+iv\rangle\right|
+|u|^2-|v|^2
}{2}.
\]
By the Cauchy--Schwarz inequality,
\begin{align*}
\left|\langle u+iv,u+iv\rangle\right|^2=
(|u|^2-|v|^2)^2+4\langle u,v\rangle^2&\leq
(|u|^2-|v|^2)^2+4|u|^2|v|^2\\
&=
(|u|^2+|v|^2)^2.
\end{align*}
Hence 
\[
p^2
\leq
\frac{|u|^2+|v|^2+|u|^2-|v|^2}{2}
=|u|^2
\]
and the proof is complete.
\epf

\

\bpf[Proof of Lemma \ref{l:decay}]
By the explicit expression for the fundamental solution,
$$
\big((P_{\theta,\rm free}-1)^{-1} f\big)(x) = \int_{\Rea^d} G_k(\gamma_\theta(x) -\gamma_\theta(y)) f(y)
\prod_{j=1}^d (1 + i (\tan \theta)F_j''(y_j))
\, dy
$$
where
\begin{equation}\label{e:green}
 G_k(z):=\frac{i k^2}{4}
 \left(\frac{k}{2\pi\rho(z)}\right)^\nu
 H_\nu^{(1)}\left(k r(z) \right),\quad r(z):=\Big( \sum_{j=1}^d z_j^2 \Big)^{1/2}, \quad\tand \nu:=(d-2)/2.
\end{equation}
By the Schur test (see, e.g., \cite[\S I.1]{Grafakos2008}) or the Riesz-Thorin interpolation theorem (see, e.g., \cite[Theorem 1.34]{Grafakos2008}),
\begin{align*}
&\| \chi_1(P_{\theta,\rm free}-1)^{-1} \chi_2\|_{L^2\to L^2}\\
&\hspace{1cm}\leq \Big({\rm ess sup}_{x\in \Rea^d} \int_{\Rea^d}\big|\chi_1(x) G_k(\gamma_\theta(x) -\gamma_\theta(y))\chi_2(y)\big| \, dy\Big)^{1/2} 
\\
&\hspace{3cm}\times\Big({\rm ess sup}_{y\in \Rea^d} \int_{\Rea^d}\big|\chi_1(x) G_k(\gamma_\theta(x) -\gamma_\theta(y))\chi_2(y)\big| \, dx\Big)^{1/2}.
\end{align*}
Now 
\beq\label{e:asymptotics1}
H_\nu^{(1)}(t) = \sqrt{\frac{2}{\pi t}}\, \exp \Big(i\Big(t - \frac{\pi \nu}{2} - \frac{\pi}{4}\Big)\Big)\bigg(1+ \mathcal{O}\bigg(\frac{1}{t}\bigg)\bigg)
\quad\tas |t|\to\infty
\eeq
for $-\pi+\delta\leq \arg t \leq \pi -\delta$; see, e.g., \cite[\S10.17(i)]{Di:26}). Therefore 
given $c>0$ there exists $C>0$ such that if $\Im r(z)\geq c$ then 
$$
|G_k(z)| \leq C k^{\frac{d+1}{2}}|r(z)|^{-\frac{d-1}{2}}\exp( -k \Im r(z) ).
$$
With $\gamma_\theta$ defined by \eqref{e:gamma}, by direct calculation
$$
\gamma_\theta(x) -\gamma_\theta(y) = (x-y) + i(\tan\theta) \big( \nabla F(x) - \nabla F(y)\big).
$$
By convexity of $F$, $\langle \nabla F(x) - \nabla F(y), x-y\rangle\geq 0.$ Therefore, by Lemma \ref{l:sneaky}, with $u=x-y$ and $v=(\tan \theta)(\nabla F(x) - \nabla F(y))$,
\beq\label{e:train1}
\Im \Big( 
\sum_{j=1}^d (\gamma_\theta(x) -\gamma_\theta(y))_j^2\Big)^{1/2} \geq  (\tan\theta)
\Big\langle \frac{x-y}{|x-y|}, \nabla F(x)-\nabla F(y)\Big\rangle.
\eeq
Since $\supp \chi_1\subset Q_{\tr}\setminus Q_{\lin}$ and $\supp \chi_2\subset Q_{\start}$, for $x\in \supp \chi_1$ and $y\in \supp \chi_2$, there is $i$ such that $|x_i-y_i|\geq c|x-y|$ and $|x_i|>R_{\lin,i}$ and $|y_i|\leq R_{\start, i}$ and hence, using that $D_x^2F=\operatorname{diag}(F_i''(x_i))$ and $F_i''(x_i)\geq 0$, we obtain
\begin{align*}
\Big\langle \frac{x-y}{|x-y|}, \nabla F(x)-\nabla F(y)\Big\rangle&=|x-y|\int_0^1\big\langle D_x^2F(sx+(1-s)y) \tfrac{x-y}{|x-y|},\tfrac{x-y}{|x-y|}\big\rangle ds\\
&\geq c|x-y|\int_0^1 F_i''(sx_i+(1-s)y_i) ds\\
& = c|x-y|\frac{F'_i(x_i)- F'_i(y_i)}{x_i-y_i}=c|x-y| \frac{x_i}{x_i-y_i}\geq c'|x-y|.
\end{align*}
The result then follows after using elliptic regularity for $P_{\theta,\rm{free}}$.
\epf

\section{Results about the complex-scaled resolvent}

\begin{theorem}[Agreement away from scaling]\label{t:agree}
If 
$\chi \in C^\infty_c(Q_{\start})$ with $\supp(1-\chi)\cap Q_{\bb}=\emptyset$, then 
$$
\chi (P_\theta-1)^{-1}\chi= \chi R_P \chi.
$$
\end{theorem}

\bpf
This is proved for the case of scattering by a Dirichlet obstacle in 
\cite[Proposition IX.2]{Ki:09}, \cite[Proof of Theorem 5.4]{BrPa:13}. For completeness we give the (short) proof here in the black-box setting. 

Given $g\in \cH$, let $u\in \cD$ solve 
$
(P_{\theta}-1)u=\chi g.
$
We first show that away from the black-box $u$ can be written as a solution to the free complex-scaled problem.
Let $\phi\in C^\infty(\mathbb{R}^d;[0,1])$ be such that $\supp \phi\cap \supp (P_{\theta}-(-k^{-2}\Delta_{\theta}))=\emptyset$ and $\supp(\nabla\phi) \subset Q_{\start}$.
Then
\begin{align*}
(P_{\theta,\rm{free}}-1)\phi u =
(P_{\theta}-1)\phi u 
 =\phi \chi g + [P_\theta,\phi]u&=\phi \chi g + [-k^{-2}\Delta_{\theta},\phi]u\\
 &=\phi \chi g + [-k^{-2}\Delta,\phi]u.
\end{align*}
so that
\beq\label{e:idea1}
\phi u=(P_{\theta,\rm{free}}-1)^{-1}\big(\phi \chi g+[-k^{-2}\Delta,\phi]u\big).
\eeq
Let $\tilde{\phi}\in C^\infty(\mathbb{R}^d;[0,1])$ with $\supp (1-\tilde{\phi})\cap \supp \phi=\emptyset$ and $\supp \tilde{\phi}\subset \mathbb{R}^d\setminus Q_{\bb}$ and $\tilde{\chi}_j\in C_c^\infty(Q_{\start}))$, $j=1,2$ with $\supp (1-\tilde{\chi}_j)\cap  \supp\chi=\supp(1-\tilde{\chi}_1)\cap \supp \nabla \phi=\supp(1-\widetilde{\chi}_2)\cap \supp \nabla \widetilde{\phi}=\supp \widetilde{\chi}_2\cap \supp \phi=\emptyset$ 

Observe that 
$$
(P_{\theta,\rm{free}}-1)^{-1}\tilde{\chi}_1:=(-k^{-2}\Delta_\theta-1)^{-1}\tilde{\chi}_1 = \big(R_{P_{\rm free}}\tilde{\chi}_1\big) |_{\Gamma};
$$
in the last equality, we mean the analytic continuation of $R_{P_{\rm free}}\tilde{\chi}_1$ to $\mathbb{C}^d$ (away from $\supp \tilde{\chi}_1$), restricted to $\Gamma$. 
Hence, since $\Gamma=\mathbb{R}^d$ on $\supp \tilde{\chi}_2$, we obtain
$$
\tilde{\chi}_2(P_{\theta,\rm{free}}-1)^{-1}\tilde{\chi}_1= \tilde{\chi}_2(-k^{-2}\Delta-1)^{-1}\tilde{\chi}_1|_{\Gamma}= \tilde{\chi}_2R_{P_{\rm free}}\tilde{\chi}_1.
$$
In particular, by \eqref{e:idea1},
\begin{align}\nonumber
0=\tilde{\chi}_2\phi u
&=\tilde{\chi}_2(P_{\theta,\rm{free}}-1)^{-1}\tilde{\chi}_1\big(\phi \chi g+[-k^{-2}\Delta,\phi]u\big)
\\
&=\tilde{\chi}_2R_{P_{\rm free}}\tilde{\chi}_1\big(\phi \chi g+[-k^{-2}\Delta,\phi]u\big).
\label{e:goodSupport}
\end{align}
Now define
$$
\tilde{u}:= (1-\phi)u+\tilde{\phi}R_{P_{\rm free}}\big(\phi \chi g+[-k^{-2}\Delta,\phi]u\big)
$$
and observe that 
\begin{align*}
(P-1)\tilde{u}&= (P-1)(1-\phi)u+(P-1)\tilde{\phi}R_{P_{\rm free}}\big(\phi \chi g+[-k^{-2}\Delta,\phi]u\big)\\
&= (P-1)(1-\phi)u+(P_{\rm free}-1)\tilde{\phi}R_{P_{\rm free}}\big(\phi \chi g+[-k^{-2}\Delta,\phi]u\big)\\
&= \chi g+[-k^{-2}\Delta,\tilde{\phi}]\widetilde{\chi}_2R_{P_{\rm free}}\widetilde{\chi_1}\big(\phi \chi g+[-k^{-2}\Delta,\phi]u\big)\\
&=\chi g,
\end{align*}
where we used~\eqref{e:goodSupport} in the last line. 
Since $\tilde{u}$ is outgoing, 
the result now follows from uniqueness of outgoing solutions to $(P-1)u=\chi g$ (see~\eqref{e:ucp}). 
\epf

\begin{lemma}[Bounding the complex-scaled resolvent in terms of the resolvent]
\label{l:inherit0}
Given $\delta>0$ there exists $k_0>0$ such that the following is true. 
Given $\chi \in C^\infty_c(Q_{\start})$
with $\supp(1-\chi)\cap 
Q_{\bb}=\emptyset$, there exists $C>0$ such that for all $\theta\in [\delta,\frac\pi2-\delta]$, and $k\geq k_0$, 
$(P_\theta-1)^{-1}: \cH\to \cD$ exists and satisfies 
$$
\|(P_\theta-1)^{-1}\|_{\cH\to \cD} \leq C \| \chi R_P \chi\|_{\cH\to \cH}.
$$
\end{lemma}

\bpf
Given Theorem \ref{t:agree}, the proof of Lemma \ref{l:inherit0} under the assumption that $(P_\theta-1)^{-1}: \cH\to \cD$ exists 
is  identical to the proof of the analogue of this result for radial PML in \cite[Proof of Lemma 3.3]{GLS2} -- the only difference is that the appropriate cut-off functions in \cite[Proof of Lemma 3.3]{GLS2}, whose support properties are defined via balls, must now have support properties defined via suitable hyperrectangles.

To demonstrate that $(P_\theta-1)^{-1}: \cH\to \cD$ exists, we repeat the argument above for the  adjoint  
$(P^*_\theta-1):  \cH\to \cD^*$, and then use the fact that an operator is invertible if and only if both the operator and its adjoint are bounded below.
\epf

\section{Results about the truncated complex-scaled resolvent}

\begin{lemma}[Estimate near PML boundary using propagation and reflection]\label{l:propagate}
Given $\epsilon,\delta, C_0>0$
there exists $k_0>0$ such that given 
$N>0$ there exists $C>0$ such that the following is true. 
Let 
$R_{\start,i}+\e<R_{i}<R_{\tr,i}$ and let $U:=Q_{\tr}\setminus \prod_{i=1}^d(-R_{i},R_{i})$ and $\psi\in C_c^\infty(\overline{U})$ with 
$$
\|\partial^\alpha \psi\|_{L^\infty}\leq C_0,\qquad |\alpha|\leq 2.
$$ 
Then for 
all $\theta\in [\delta, \pi/2-\delta]$, all $k\geq k_0$
and all 
$u \in \cD(Q_{\tr})$ with $u=0$ on $\partial Q_{\tr}$ 
\begin{equation}\label{e:propagate}
\|\psi u\|_{\cD(Q_{\tr})}\leq C \Big(k \| (P_\theta^D-1)u\|_{\cH(Q_{\tr})} + k^{-N} \| u\|_{\cH(Q_{\tr})}\Big),
\end{equation}
and for all $v\in \cH(Q_\tr)$ 
\begin{equation}\label{e:propagate2}
\|\psi v\|_{\cH(Q_{\tr})}\leq C \Big(k \| (P_\theta^D-1)^*v\|_{\cD^*(Q_{\tr})} + k^{-N} \| v\|_{\cH(Q_{\tr})}\Big).
\end{equation}
\end{lemma}

\bpf
To prove~\eqref{e:propagate}, let $f:= (P_\theta^D-1)u\in \cH(Q_{\tr})$. 
By~\cite[Lemma 3.6]{GGGLS2}, 
since $R_{\tr,i}> R_{\lin,i}$,
we may extend $u$ to a solution $\tilde{u}$ to 
$$
(P_\theta-1)\tilde{u}=\tilde{f}
$$
in a neighborhood, $V$, of $\partial Q_{\tr}$ with $U\Subset V$, where $\tilde{f}\in L^2(V)$ is an appropriate reflection of $f$ across $\partial Q_{\tr}$ and hence $\|\tilde{f}\|_{L^2(V)}\leq C\|f\|_{\cH}$.

We now apply the propagation and elliptic estimates of \cite[Theorem E.47, Theorem E.33]{DyZw:19}. 
To do this, we first show that for any $(x,\xi)\in T^*U$ with $p_{\theta,\rm fr}(x,\xi)=1$ the backward Hamiltonian trajectory for $\Re p_{\theta,\rm fr}$ encounters $p_{\theta,\rm fr}\neq 1$ before exiting $T^*V$. Once we have proved this, the result follows directly by applying to $\tilde{u}$~\cite[Theorem E.47]{DyZw:19} near points where $p_{\theta,\rm fr}-1=0$ and~\cite[Theorem E.33]{DyZw:19} otherwise.

To prove the dynamical statement, observe that for any $j=1,\dots,d$ and $(x,\xi)\in T^*U$, with $x_j> R_{\start, j}$ and  $p_{\theta,\rm fr}(x,\xi)=1$, we have
$
\xi_j=0 \Rightarrow \partial_{\xi_j}\Re p_{\theta, \rm fr}=0
$
and hence $x_j$ is constant along the Hamiltonian trajectory of $\Re p_{\theta,\rm fr}$ until the trajectory meets $p_{\theta,\rm fr}\neq 1$. This implies that such trajectories remain inside $V$ until they either encounter $p_{\theta,\rm fr}\neq 1$ or exit $V$. To see that the former happens notice that $\sum_{i\neq j}\xi_i^2>c>0$ and hence there is $i$ with $|\xi_i|>0$, $F_i''(x_i)=0$. This implies that $x_i$ increases (or decreases) and $\xi_i$ is constant along the Hamiltonian flow for $\Re p_{\theta,\rm fr}$ until $F_i''(x_i)\neq 0 $. Since $F_i''(x_i)\neq 0$, $\xi_i\neq 0$ implies $p_{\theta,\rm fr}(x_i,\xi_i)\neq 1$ and $F_i''(x_i)>0$ for $|x_i|>R_{\start,i}$ this implies that the backward trajectory for $\Re p_{\theta,\rm free}$ meets $p_{\theta,\rm fr}\neq 1$ before exiting $V$. 

For~\eqref{e:propagate2}, the procedure is the same except that we set $f:=(P_\theta^D-1)^*v\in\cD^*(Q_{\tr})$. Then, in a neighborhood, $V$, of $\partial Q_{\tr}$, 
$$
(P_\theta-1)^*\tilde{v}=\tilde{f}
$$
where $\tilde{f}$ in $H^{-2}(V)$ and $\|\tilde{f}\|_{H_k^{-2}(V)}\leq C\|f\|_{\cD^*}$ (see the proof of~\cite[Lemma 3.2]{GGGLS2}). The estimate then follows in the same way except that we follow trajectories of $\Re p_{\theta,\rm fr}$ forward. 
\epf

\begin{lemma}[Bounding the truncated complex-scaled resolvent in terms of the resolvent]
\label{l:inherit}
Suppose that there 
exists 
 $\chi \in C^\infty_c(Q_{\start})$
with $\supp(1-\chi)\cap 
Q_{\bb}=\emptyset$,  $\mathcal{J}\subset (0,\infty)$, and  $M>0$ such that for $k_0>0$ 
and there is $C_0>0$ such that for 
all $k\in\mathcal{J}$, $k \geq k_0$
$$
\| \chi R_P \chi\|_{\cH\to \cD}\leq C_0k^M.
$$
Given $\delta, \e>0$ there exists $C>0$ and $k_1\geq k_0>0$ such that for 
$R_{\start,i}+\e<R_{\tr,i}$,
$\theta\in [\delta, \pi/2-\delta]$, and all $k\in\mathcal{J}$ with $k \geq k_1$,
$$
\|(P_\theta^D-1)^{-1}\|_{\cH(Q_{\tr})\to \cD(Q_{\tr})}\leq C
\| \chi R_P \chi\|_{\cH\to \cD}.
$$
\end{lemma}

\bpf
Let $(P_\theta^D-1)u=g$ and $(P_\theta^D-1)^*v=g$.
Recall that an operator is invertible if and only if both the operator and its adjoint are bounded below. 
By this and
Lemma \ref{l:inherit0}, it is sufficient to prove that 
\beq\label{e:stp1}
\|u \|_{\cD(Q_{\tr})}\leq C \|(P_\theta-1)^{-1}\|_{\cH\to \cD}\|g\|_{\cH(Q_{\tr})}
\eeq
and 
\beq\label{e:stp2}
\|v \|_{\cH(Q_{\tr})}\leq C \|(P_\theta^*-1)^{-1}\|_{\cD^*\to \cH}\|g\|_{\cD^*(Q_{\tr})}.
\eeq
Let $U$ be as in Lemma \ref{l:propagate}  and let $\chi \in C^\infty_c(Q_{\tr})$ 
with $\supp \partial \chi \subset U$ and $\supp(1-\chi)\cap Q_{\bb}=\emptyset$.
By \eqref{e:propagate}, given $N>M+1$, there exists $C>0$ such that
\beq\label{e:add1}
\|(1-\chi)u\|_{\cD(Q_\tr)}\leq C \Big(k \|g\|_{\cH(Q_\tr)} + k^{-N} \| u\|_{\cH(Q_{\tr})}\Big).
\eeq
On the other hand, since $\chi \in C^\infty_c(Q_{\tr})$, 
$$
\chi u = (P_\theta-1)^{-1}\Big( \chi g + [-k^{-2}\Delta_\theta, \chi]u\Big)
$$
so that, since $\supp \partial \chi \subset U$, 
\begin{align}\nonumber
\|\chi u\|_{\cD} 
&\leq C \|(P_\theta-1)^{-1}\|_{\cH\to \cD}\Big( \| g\|_{\cH(Q_\tr)} + k^{-1} \| u\|_{H^1_k(U)}\Big) \\
&\leq C \|(P_\theta-1)^{-1}\|_{\cH\to \cD}\Big( \| g\|_{\cH(Q_\tr)} + k^{-N} \| u\|_{\cH(Q_{\tr})}\Big),\label{e:add2}
\end{align}
where we have used \eqref{e:propagate} again in the last line. 
By adding \eqref{e:add1} and \eqref{e:add2} and using that $\|(P_\theta-1)^{-1}\|_{\cH\to\cD}$ is both polynomially bounded and bounded below by $ck$ (see, e.g., \cite[Remark 3.4]{GLS2})
we obtain that 
$$
\|u \|_{\cH(Q_{\tr})}\leq C \|(P_\theta-1)^{-1}\|_{\cH\to \cD}\|g\|_{\cH(Q_{\tr})}.
$$
The bound \eqref{e:stp1}
then follows from
the definition of $\|\cdot\|_{\cD(Q_\tr)}$ \eqref{e:Dnorm}.

Repeating this argument for the adjoint problem using~\eqref{e:propagate2} instead of~\eqref{e:propagate} proves \eqref{e:stp2}, and the proof is complete.
\epf

\section{Proof of Theorem \ref{t:main}}

By Theorem \ref{t:agree}, it is sufficient to prove that, for $\chi\in C_c^\infty(Q_{\start})$ with 
 $\supp(1-\chi)\cap Q_{\bb}=\emptyset$, 
\beq\label{e:main2}
\big\|\chi \big((P^D_\theta-1)^{-1}-
(P_\theta-1)^{-1}
\big)\chi\big\|_{\cH\to \cD}\leq C\exp \big(- c k(\tan \theta)\min_i\{R_{\tr,i}-R_{\start,i}\}\big).
\eeq
Given $g\in L^2$ with $\supp g\subset \Omega_{\operatorname{tr}}$, let $u\in \cD$ solve 
$$
(P_{\theta}-1)u=\chi g.
$$
As in the proof of Theorem \ref{t:agree}, $u$ can be written as a solution to the free complex-scaled problem:~let $\phi\in C^\infty(\mathbb{R}^d;[0,1])$ be such that $\supp \phi\cap \supp (P_{\theta}-(-k^{-2}\Delta_{\theta}))=\emptyset$ and $\supp(\nabla\phi) \subset Q_{\start}$.
By \eqref{e:idea1}, 
\beq\label{e:idea2}
\phi u=(P_{\theta,\rm{free}}-1)^{-1}\big(\phi \chi g+[-k^{-2}\Delta_\theta,\phi]u\big).
\eeq
The idea of the proof is now to show that 
if $\psi$ is a cut-off function supported in $Q_{\tr}$ with $\psi\equiv 1$ except in a small neighborhood of $\partial Q_{\tr}$ then 
$\psi u$ is a good approximation to $(P^D_\theta-1)^{-1}\chi g$.
Furthermore, the commutator in this difference is supported near $\partial Q_{\tr}$, and thus the error decreases exponentially via 
the expression \eqref{e:idea1} and the bounds on $(P_{\theta,\rm{free}}-1)^{-1}$ in Lemma \ref{l:decay}.

In more detail, let $\psi\in C_c^\infty(Q_{\tr})$ with $\supp(1-\psi)\cap \supp \chi=\emptyset$ and 
$\supp (1-\psi)\cap 
\prod_{i=1}^d (-R_{\tr,i} +\e, R_{\tr,i}-\e)
=\emptyset$ and 
$\supp \partial\psi\cap \supp (1-\phi)=\emptyset$.
Then $\psi u$ satisfies $\psi u|_{\partial Q_{\tr}}=0$ and 
\begin{align*}
(P_\theta-1)\psi u&=\psi \chi g+[-k^{-2}\Delta_{\theta},\psi]u\\
&=\chi g+[-k^{-2}\Delta_{\theta},\psi]\phi u\\
&= \chi g+[-k^{-2}\Delta_{\theta},\psi](P_{\theta,\rm free}-1)^{-1}\big(\phi \chi g+[-k^{-2}\Delta_\theta,\phi]u\big).
\end{align*}
Let $\widetilde{\chi}\in C_c^\infty(\mathbb{R}^d)$ with $\supp \widetilde\chi\subset Q_{\start}$ with $\supp (1-\widetilde{\chi})\cap (\supp \chi \cup \supp \nabla \phi)=\emptyset$ and let $\widetilde{\phi}\in C_c^\infty(Q_{\tr};[0,1])$ with $\supp \widetilde\phi \cap \prod_{i=1}^d (-R_{\tr,i}+\epsilon,R_{\tr,i}-\epsilon)=\emptyset$ and $\supp (1-\widetilde{\phi})\cap \supp \nabla \psi=\emptyset$. Then
\begin{align*}
&(P^D_\theta-1)^{-1}\chi g-\psi u\\
&=-(P^{D}_\theta-1)^{-1}[-k^{-2}\Delta_{\theta},\psi](P_{\theta,\rm free}-1)^{-1}\big(\phi \chi g+[-k^{-2}\Delta_\theta,\phi]u\big)\\
&=-(P^{D}_\theta-1)^{-1}[-k^{-2}\Delta_{\theta},\psi]\widetilde{\phi}(P_{\theta,\rm free}-1)^{-1}\widetilde{\chi}\big(\phi \chi g+[-k^{-2}\Delta_\theta,\phi](P_\theta-1)^{-1}\chi g\big).
\end{align*}
By the combination of Lemmas \ref{l:inherit0} and \ref{l:inherit}, $(P^{D}_\theta-1)^{-1}$ is polynomially bounded. 
We now apply Lemma \ref{l:decay} to bound 
$\| \widetilde{\phi} (P_{\theta,\rm free}-1)^{-1}\widetilde{\chi} \|_{L^2 \to H_k^2}$. 
By the definitions of 
$\widetilde{\phi}$ and 
$\widetilde{\chi}$, for $\e>0$ small enough,
$${\rm dist}(\supp \widetilde{\phi}, \supp \widetilde{\chi})\geq \min_i\{ R_{\tr,i}-R_{\start,i}-\e\}\geq \frac{1}{2}\min_i\{R_{\tr,i}-R_{\start,i}\}.
$$
Therefore, by Lemma \ref{l:decay}, 
given $\delta,k_0>0$ there exist $C,\tilde{c}>0$ such that for all $\theta\in [\delta,\pi/2-\delta]$, and all $k\geq k_0$, 
$$
\big\| \widetilde{\phi} (P_{\theta,\rm free}-1)^{-1}\widetilde{\chi} \big\|_{L^2(\Rea^d) \to H_k^2(\Rea^d)} \leq C \exp\big(-\tilde{c}k\tan \theta\min_i \{R_{\tr,i}-R_{\start,i}\}\big)
$$
and the proof is complete.

\section*{Acknowledgements}

JG and ES were supported by  ERC Synergy Grant PSINumScat (101167139).
JG was additionally supported by EPSRC grants EP/V001760/1 and EP/V051636/1 and  Leverhulme Research Project Grant RPG-2023-325.

ChatGPT version 5.6 Sol 
was used in the writing of Lemma \ref{l:sneaky}, and version 6 Astra was used to proofread the document and create Figure \ref{fig:Martin_is_better_than_chatGPT}.

\bibliographystyle{alpha}
\bibliography{references.bib}

@book{DyZw:19,
  author    = {Dyatlov, S. and Zworski, M.},
  title     = {Mathematical Theory of Scattering Resonances},
  series    = {Graduate Studies in Mathematics},
  volume    = {200},
  publisher = {American Mathematical Society},
  address   = {Providence, RI},
  year      = {2019},
  doi       = {10.1090/gsm/200}
}

@article{GLS2,
  author  = {Galkowski, J. and Lafontaine, D. and Spence, E. A.},
  title   = {{Perfectly matched layer} truncation is exponentially accurate
             at high frequency},
  journal = {SIAM J. Math. Anal.},
  volume  = {55},
  number  = {4},
  pages   = {3344--3394},
  year    = {2023},
  doi     = {10.1137/21M1443716}
}

@article{BrPa:13,
  author  = {Bramble, J. H. and Pasciak, J. E.},
  title   = {Analysis of a Cartesian {PML} approximation to acoustic
             scattering problems in {$\mathbb R^2$} and {$\mathbb R^3$}},
  journal = {J. Comput. Appl. Math.},
  volume  = {247},
  pages   = {209--230},
  year    = {2013},
  doi     = {10.1016/j.cam.2012.12.022}
}

@misc{GGGLS2,
  author = {Galkowski, J. and Gong, S. and Graham, I. G. and
            Lafontaine, D. and Spence, E. A.},
  title  = {Convergence of overlapping domain decomposition methods with
            {PML} transmission conditions applied to nontrapping
            {Helmholtz} problems},
  year   = {2024},
  note   = {arXiv:2404.02156},
  doi    = {10.48550/arXiv.2404.02156}
}

@article{SjZw:91,
  author = {Sj\"{o}strand, J. and Zworski, M.},
  title = {Complex scaling and the distribution of scattering poles},
  journal = {J. Amer. Math. Soc.},
  fjournal = {Journal of the American Mathematical Society},
  volume = {4},
  year = {1991},
  number = {4},
  pages = {729--769},
  issn = {0894-0347},
  mrclass = {35P25 (35B20 58G25)},
  mrnumber = {1115789},
  mrreviewer = {George D. Raikov},
  doi = {10.2307/2939287},
  url = {https://doi.org/10.2307/2939287},
}

@article{HoScZs:03,
  title = {{Solving time-harmonic scattering problems based on the pole condition II: convergence of the PML method}},
  author = {Hohage, T. and Schmidt, F. and Zschiedrich, L.},
  journal = {SIAM Journal on Mathematical Analysis},
  volume = {35},
  number = {3},
  pages = {547--560},
  year = {2003},
  publisher = {SIAM},
}

@article{BrPa:07,
  title = {{Analysis of a finite PML approximation for the three dimensional time-harmonic Maxwell and acoustic scattering problems}},
  author = {Bramble, J. H. and Pasciak, J. E.},
  journal = {Mathematics of Computation},
  volume = {76},
  number = {258},
  pages = {597--614},
  year = {2007},
}

@article{LaSo:98,
  title = {{On the existence and convergence of the solution of PML equations}},
  author = {Lassas, M. and Somersalo, E.},
  journal = {Computing},
  volume = {60},
  number = {3},
  pages = {229--241},
  year = {1998},
  publisher = {Springer},
}

@article{LaSo:01,
  title = {{Analysis of the PML equations in general convex geometry}},
  author = {Lassas, M. and Somersalo, E.},
  journal = {Proceedings of the Royal Society of Edinburgh Section A: Mathematics},
  volume = {131},
  number = {5},
  pages = {1183--1207},
  year = {2001},
  publisher = {Royal Society of Edinburgh Scotland Foundation},
}

@article{Be:94,
  author = {J.-P. Berenger},
  title = {A perfectly matched layer for the absorption of electromagnetic waves},
  journal = {J. Comp. Phys.},
  year = {1994},
  volume = {114},
  pages = {185-200},
  month = {October},
  acmid = {195266},
  address = {San Diego, CA, USA},
  doi = {10.1006/jcph.1994.1159},
  issn = {0021-9991},
  issue = {2},
  numpages = {16},
  publisher = {Academic Press Professional, Inc.},
  url = {http://portal.acm.org/citation.cfm?id=195261.195266},
}

@article{KiPa:10,
  title={{Analysis of a Cartesian PML approximation to acoustic scattering problems in $\mathbb{R}^2$}},
  author={Kim, S. and Pasciak, J. E.},
  journal={Journal of Mathematical Analysis and Applications},
  volume={370},
  number={1},
  pages={168--186},
  year={2010},
  publisher={Elsevier}
}

@article{KiPa:10a,
  title={{Analysis of the spectrum of a Cartesian perfectly matched layer (PML) approximation to acoustic scattering problems}},
  author={Kim, S. and Pasciak, J. E.},
  journal={Journal of Mathematical Analysis and Applications},
  volume={361},
  number={2},
  pages={420--430},
  year={2010},
  publisher={Elsevier}
}

@article{ChXi:13,
  title = {{A source transfer domain decomposition method for Helmholtz equations in unbounded domain}},
  author = {Chen, Z. and Xiang, X.},
  journal = {SIAM Journal on Numerical Analysis},
  volume = {51},
  number = {4},
  pages = {2331--2356},
  year = {2013},
  publisher = {SIAM},
}

@article{wang2026wavenumber,
  title={Wavenumber-explicit stability and preasymptotic error analysis of UPML finite element method for obstacle scattering problems},
  author={Wang, Y. and Zheng, W.},
  journal={arXiv preprint arXiv:2607.20331},
  year={2026}
}

@article{LSW1,
  title = {For most frequencies, strong trapping has a weak effect in frequency-domain scattering},
  author = {Lafontaine, D. and Spence, E. A. and Wunsch, J.},
  journal = {Communications on Pure and Applied Mathematics},
  volume = {74},
  number = {10},
  pages = {2025-2063},
  year = {2021},
}

@book{Zw:12,
  title = {Semiclassical analysis},
  author = {Zworski, M.},
  year = {2012},
  publisher = {American Mathematical Society Providence, RI},
}

@phdthesis{Ki:09,
  title={Analysis of a PML method applied to computation of resonances in open systems and acoustic scattering problems},
  author={Kim, S.},
  year={2009},
  school={Texas A\&M University}
}

@misc{Di:26,
	Howpublished = {Digital Library of Mathematical Functions, \url{http://dlmf.nist.gov/}},
	Title = {{Digital Library of Mathematical Functions}},
	Author = {{NIST}},
	Year = 2026}

@book{Grafakos2008,
  author = {Grafakos, L.},
  title = {{Classical Fourier Analysis}},
  edition = {2},
  series = {Graduate Texts in Mathematics},
  volume = {249},
  publisher = {Springer},
  address = {New York},
  year = {2008},
  doi = {10.1007/978-0-387-09432-8},
}

@article{Va:04,
  title={Propagation of singularities for the wave equation on manifolds with corners},
  author={Vasy, A.},
  journal={S{\'e}minaire {\'E}quations aux d{\'e}riv{\'e}es partielles (Polytechnique)},
  pages={1--16},
  year={2004}
}

@article{GSAN,
  title = {{Numerical analysis of the high-frequency Helmholtz equation using semiclassical analysis}},
  author = {Galkowski, J. and Spence, E. A.},
  journal = {Acta Numerica},
  volume = {35},
  pages = {1--171},
  year = {2026},
  publisher = {Cambridge University Press},
}
\end{document}